\documentclass[twoside]{article}
\usepackage{amsfonts,latexsym,amsmath,amsthm,amscd,eucal,amssymb}

\catcode`\@=11
\long\def\@makefntext#1{
\protect\noindent \hbox to 3.2pt {\hskip-.9pt  
$^{{\eightrm\@thefnmark}}$\hfil}#1\hfill}               

\def\@makefnmark{\hbox to 0pt{$^{\@thefnmark}$\hss}}    
        
\def\ps@myheadings{\let\@mkboth\@gobbletwo
  \def\@oddhead{{\slshape\rightmark}\hfil{\footnotesize\thepage}}
  \def\@oddfoot{}
  \def\@evenhead{{\footnotesize\thepage}\hfil\slshape\leftmark}
  \def\@evenfoot{}
  \def\sectionmark##1{}\def\subsectionmark##1{}
}

\oddsidemargin=\evensidemargin
\renewcommand\section{\@startsection {section}{1}{\z@}%
                                   {-3.5ex \@plus -1ex \@minus -.2ex}%
                                   {2.3ex \@plus.2ex}%
                                   {\tenbf\large\bfseries}}
\renewcommand\subsection{\@startsection{subsection}{2}{\z@}%
                                     {-3.25ex\@plus -1ex \@minus -.2ex}%
                                     {1.5ex \@plus .2ex}%
                                     {\normalfont\bfseries}}

\newcommand{\textlineskip}{\baselineskip=13pt}

\newcommand{\timenow}{%
  \@tempcnta=\time \divide\@tempcnta by 60 \number\@tempcnta:\multiply
  \@tempcnta by 60 \@tempcntb=\time \advance\@tempcntb by -\@tempcnta
  \ifnum\@tempcntb <10 0\number\@tempcntb\else\number\@tempcntb\fi}

\def\abstracts#1#2#3{{
        \centering{\begin{minipage}{4.5in}\footnotesize\baselineskip=10pt
        \parindent=0pt #1\par 
        \parindent=15pt \footnotesize\baselineskip=10pt
        {\footnotesize\it Keywords}\/: #2\par
        \parindent=15pt \footnotesize\baselineskip=10pt
        {\footnotesize\it 1991 MSC}\/: #3
        \end{minipage}}\par}} 

\renewenvironment{thebibliography}[1]
        {\frenchspacing
         \ninerm\baselineskip=11pt
         \begin{list}{\arabic{enumi}.}
        {\usecounter{enumi}\setlength{\parsep}{0pt}     
         \setlength{\leftmargin 17pt}{\rightmargin 0pt}   
         \setlength{\itemsep}{0pt} \settowidth
        {\labelwidth}{#1.}\sloppy}}{\end{list}}

\def\pmb#1{\setbox0=\hbox{#1}
        \kern-.025em\copy0\kern-\wd0
        \kern.05em\copy0\kern-\wd0
        \kern-.025em\raise.0433em\box0}

\def\fnt#1#2{\footnotetext{\kern-.3em
        {$^{\mbox{\scriptsize #1}}$}{#2}}}

\def\runninghead#1#2{\pagestyle{myheadings}
\markboth{{\protect\footnotesize\it{\quad #1}}}
{{\protect\footnotesize\it{#2\quad}}}}
\font\tenbf=cmbx10

\font\ninerm=cmr9

\font\eightrm=cmr8

\newtheorem{theorem}{Theorem}
\newtheorem{assumption}[theorem]{Assumption}
\newtheorem{proposition}[theorem]{Proposition}
\newtheorem{definition}[theorem]{Definition}
\newtheorem{corollary}[theorem]{Corollary}
\newtheorem{lemma}[theorem]{Lemma}

\newtheorem{remark}{Remark}

\def\bA{\mathbb A}
\def\bF{\mathbb F}
\def\bG{\mathbb G}

\def\bW{\mathbb W}
\def\calf{{\mathcal F}}
\def\caln{{\mathcal N}}
\def\F{\mathcal F}
\def\cX{\mathcal X}
\def\cZ{\mathcal Z}

\DeclareMathOperator{\Det}{Det}
\def\ee{\mathrm{e}}

\def\cC{\mathcal C}
\def\u{\mathbf u}
\def\v{\mathbf v}
\def\f{\mathbf f}
\def\g{\mathbf g}
\def\h{\mathbf h}
\def\w{\mathbf w}
\def\E{\mathbb E}

\def\Chi{\mathcal X}

\def\Pr{\mathbb P}

\def\R{\mathbb R}

\def\linea{\vskip 1\baselineskip}

\begin{document}
\setlength{\textheight}{7.7truein}

\runninghead{S. Bonaccorsi, F. Confortola, E. Mastrogiacomo}{Control of stochastic differential equations with dynamical boundary conditions}

\normalsize\textlineskip
\thispagestyle{empty}
\setcounter{page}{1}


\vspace*{1in}

\centerline{\bf Optimal control of stochastic differential equations}
\baselineskip=13pt
\centerline{\bf with dynamical boundary conditions}
\baselineskip=13pt
\vspace*{0.37truein}

\centerline{\footnotesize Stefano BONACCORSI\footnote{stefano.bonaccorsi@unitn.it}, \qquad Fulvia CONFORTOLA\footnotemark, 
  \qquad Elisa MASTROGIACOMO}
\baselineskip=12pt
\centerline{\footnotesize\it Dipartimento di Matematica, Universit\`a di Trento,}
\baselineskip=10pt
\centerline{\footnotesize\it via Sommarive 14, 38050 Povo (Trento), Italia}
\baselineskip=10pt
\footnotetext{Current address: {fulvia.confortola@unimib.it}}
\vspace*{0.21truein} \abstracts {In this paper we investigate the
  optimal control problem for a class of stochastic Cauchy evolution
  problem with non standard boundary dynamic and control. The model is
  composed by an infinite dimensional dynamical system coupled with a
  finite dimensional dynamics, which describes the boundary conditions
  of the internal system. In other terms, we are concerned with non
  standard boundary conditions, as the value at the boundary is
  governed by a different stochastic differential equation.  }
{Stochastic differential equations in infinite dimensions, 
dynamical boundary conditions, optimal control}
{}

\vspace*{4pt}
\baselineskip=13pt
\normalsize     


\section{Setting of the problem}

Our model is a one dimensional semilinear diffusion equation in a
confined system, where interactions with extremal points cannot be
disregarded. The extremal points have a mass and the boundary
potential evolves with a specific dynamic. Stochasticity enters
through fluctuations and random perturbations both in the inside as on
the boundaries; in particular, in our model we assume that the control
process is perturbed by a noisy term.

There is a growing literature concerning such
problems; we shall mention the paper \cite{chueshov/schmalfuss} where
a problem in a domain ${\mathcal O}\subset \R^n$ is concerned; the authors
cite as an example an SPDE with stochastic perturbations which appears
in connection with random fluctuations of the atmospheric pressure
field. As opposite to ours, however, that paper is not concerned with
control problems. Quite recently, the authors became aware of the
paper \cite{bertini/noja/posilicano} where a different application to
some generalized Lamb model is proposed.

The internal dynamic is described by a stochastic evolution problem in
the unit interval $D = [0,1]$
\begin{equation}\label{dyn-int-intro}
  \partial_t u(t,x) = \partial_x^2 u(t,x)
  + f(t,x,u(t,x)) + g(t,x,u(t,x)) \dot W(t,x)
\end{equation}
which we write as an abstract evolution problem on the space $L^2(0,1)$
\begin{equation}
  {\rm d} u(t) = A_m u(t) + F(t,u(t)) \, {\rm d}t + G(t,u(t)) \, {\rm d}W(t),
\end{equation}
where the leading operator is $A_m = \partial_x^2$ with domain $D(A_m)
= H^2(0,1)$. We assume that $f$ and $g$ are real valued mappings,
defined on $[0,T] \times [0,1] \times \R$, which verify some
boundedness and Lipschitz continuity assumptions.

The boundary dynamic is governed by a finite dimensional system which
follows a (ordinary, two dimensional) stochastic differential equation
\begin{equation*}
  \partial_t v_i(t) = -b_i v_i(t) + \partial_\nu u(t,i) + h_i(t) \dot V_i(t), \qquad i=0,1
\end{equation*}
where $b_i$ are positive numbers and $h_i(t)$ are bounded, measurable
functions; $\partial_\nu$ is the normal derivative on the boundary,
and coincides with $(-1)^i \partial_x$ for $i=0,1$. For notational
semplicity, we introduce the $2\times 2$ diagonal matrices $B = {\rm
  diag}(-b-_0,b_1)$ and $h(t) = {\rm diag}(h_0(t),h_1(t))$. There is a
constraint
\begin{equation*}
  L u = v
\end{equation*}
which we interpret as the operator evaluating boundary conditions; the
system is coupled by the presence, in the second equation, of a
feedback term $C$ that is an unbounded operator
\begin{equation*}
  C u = \begin{pmatrix}\partial_x u(0) \\ -\partial_x u(1)\end{pmatrix}.
\end{equation*}

The idea is to write the problem in abstract form for the vector $\u =
\begin{pmatrix}u(\cdot) \\ v\end{pmatrix}$ on the space
$\Chi = L^2(0,1)\times \R^2$, that is
\begin{equation}\label{eq:intro-1}
  \begin{cases}
    {\rm d}\u = \bA\u(t) + \bF(t,\u(t)) \, {\rm d}t + \bG(t,\u(t)) \, {\rm d}\bW(t)
    \\
    \u(0) = \begin{pmatrix}u_0 \\ v_0 \end{pmatrix}
  \end{cases}
\end{equation}
Our main concern is to study spectral properties of the matrix
operator 
\begin{equation*}
  \bA = 
  \begin{pmatrix}
    A_m & 0 \\ C & B
  \end{pmatrix}
\end{equation*}
on the domain
\begin{equation*}
  D(\bA) = \{\u\in D(A_m) \times \R^{2} \,:\, L u = v\}.
\end{equation*}

\begin{theorem}\label{th:1}
  $\bA$ is the infinitesimal generator of a strongly continuous,
  analytic semigroup of contractions $\ee^{t\bA}$, self-adjoint and
  compact.
\end{theorem}

We shall prove the above theorem in Section \ref{sez.det.op}. Further,
we shall prove that $\bA$ is a self-adjoint operator with compact
resolvent, which implies that the generated semigroup is
Hilbert-Schmidt. Moreover, we can characterize the complete,
orthonormal system of eigenfunctions associated to $\bA$.

Let us fix a complete probability space $(\Omega,\F,\{\F_t\},\Pr)$; on
this space we define $W(t)$, that is a space-time Wiener process
taking values in $X$ and $V(t) = (V_1(t),V_2(t))$, that is a
$\R^2$-valued Wiener process, such that $W(t,x)$ and $V(t)$ are
independent.

As a corollary to Theorem \ref{th:1}, using standard results for
infinite dimensional stochastic differential equations, compare
\cite[Theorem 7.4]{dpz:stochastic}, we obtain the following existence
result

\begin{theorem}\label{th:2}
  For any initial condition $\begin{pmatrix}u_0 \\ v_0 \end{pmatrix}
  \in X \times \R^2$ there exists a unique process $\u \in
  L^2_F(0,T;X \times \R^2)$ such that
  \begin{equation*}
    \u(t) = \ee^{t\bA} \begin{pmatrix}u_0 \\ v_0 \end{pmatrix} 
    + \int_0^t \ee^{(t-s) \bA}\bF(\u(s)) \, {\rm d}s + \int_0^t  \ee^{(t-s) \bA}
    \bG(\u(s)) \, {\rm d}\bW(s)
  \end{equation*}
  that is by definition a mild solution of \eqref{eq:intro-1}.
\end{theorem}

\vskip1\baselineskip 

The
abstract semigroup setting we propose in this paper allows to obtain
an optimal control synthesis for the above evolution problem with
boundary control and noise.   This means
that we assume a boundary dynamics of the form:
\begin{equation}\label{bcond.controllointro}
  \partial_t v(t) = b v(t) - \partial_\nu u(t,\cdot) + h(t) [z(t) +
  \dot V(t)]
\end{equation}
where $z(t)$ is the control process and takes values in a given subset of $\R^2$. 

As before, we can write the system -- defined by the internal
evolution problem (\ref{dyn-int-intro}) and the dynamical boundary
conditions described by (\ref{bcond.controllointro}) -- in the
following abstract form
\begin{equation}\label{state equationintro}
  \begin{cases}
  {\rm d}\textbf{u}_t^z = \bA \textbf{u}_t^z \, {\rm d}t+
  \bF(t,\textbf{u}_t^z) \, {\rm d}t +\bG(t,\textbf{u}_t^z)[ Pz_t \,
  {\rm d}t + {\rm d}\mathbf{W}_t] \\
  \textbf{u}_{t_0}=\textbf{u}_0.
  \end{cases}
\end{equation}
$P: \R^2 \to \cX$ denote the immersion of the boundary
space in the product space $\cX = L^2(0,1) \times \R^2$.

The aim is to choose a control process $z$, within a set of admissible controls, in such way to minimize a cost functional of the form
\begin{equation}
  J(t_0,u_0,z)=\E \int_{t_0}^T \lambda(s,\mathbf{u}_s^z,z_s)) \, {\rm d}s + \E
  \phi(\mathbf{u}_T^z)
\end{equation}
where $\lambda$ and $\phi$ are given real functions. In our setting, altough the control lives in a finite
dimensional space, we obtain an abstract optimal control problem in
infinite dimensions. Such type of problems has been exhaustively
studied by Fuhrman and Tessitore in \cite{FT1}. The control problem is
understood in the usual weak sense (see \cite{FlSo}).  We prove that
if $f$ and $g$ are sufficiently regular then the abstract control
problem, under suitable assumptions on $\lambda$ and $\phi$, can be
solved and we can characterize optimal controls by a feedback law (see
Theorem \ref{maincontrollo} and compare Theorem 7.2 in \cite{FT1}).

\begin{theorem}
  In our assumptions, there exists an admissible control $\{\bar z_t,\ t \in [0,T]\}$ taking values in a bounded subset of $\R^2$, such that the closed loop
  equation:
  \begin{equation}
    \begin{cases}
      d\overline{\u}_\tau =
    \bA\overline{\u}_\tau\; d\tau
    +\bG(\tau,\overline{\u}_\tau)P\Gamma(\tau, \overline{\u}_{\tau},
    \bG(\tau, \overline{\u}_{\tau })^*\nabla_{x}v(\tau ,
    \overline{\u}_{\tau }))\; d\tau
    \\ 
    \phantom{d\overline{\u}_\tau =
    \bA\overline{\u}_\tau\; d\tau
    }+\bF(\tau,\overline{\u}_\tau)\;
    d\tau+\bG(\tau,\overline{\u}_\tau)\; d\mathbf{W}_\tau,\qquad
    \tau\in [t_{0},T],
    \\ 
    \overline{\u}_{t_{0}} = \u_{0}\in \Chi.
    \end{cases}
  \end{equation}
  admits a solution and the
  couple $(\overline{z},\overline{\u})$ is optimal for the control
  problem.
\end{theorem}

Stochastic boundary value problems are already present in the
literature, see the paper \cite{maslowski} and the references therein;
in those papers, the approach to the solution of the system is more
similar to that in \cite{chueshov/schmalfuss}. We also need to mention
the paper \cite{debussche/fuhrman/tessitore} for a one dimensional
case where the boundary values are set equal to a white noise
mapping.

\section{Generation properties}\label{sez.det.op}

Let $X = L^2(0,1)$ be the Hilbert space of square integrable real
valued functions defined on $D = [0,1]$ and $\Chi = X \times \R^2$. In
this section we consider the following initial-boundary value problem
on the space $\Chi$
\begin{equation}\label{sistema}
  \begin{cases}
    \frac{d}{dt} u(t) = A_m u(t) \\
    v(t) = Lu(t) \\
    \frac{d}{dt} v(t) = B v(t) - C u(t) \\
    u(0) = u_0 \in X, \quad v(0) = v_0 \in \R^2.
  \end{cases}
\end{equation}
In the above equation, $A_m$ is an unbounded operator with maximal domain 
\begin{equation*}
  A_m = \partial^2_x, \quad D(A_m) = H^2(0,1);
\end{equation*}
$B$ is a diagonal matrix with negative entries $(-b_0,-b_1)$. 

Let $C: D(C) \subset X \to \partial X$ the feedback operator,
defined on $D(C) = H^1(0,1)$ as
\begin{equation*}
  Cu = \begin{pmatrix}\partial_x u(0) \\ -\partial_x u(1)\end{pmatrix}.
\end{equation*}

The boundary evaluation operator $L$ is the mapping $L:X \to \R^2$
given by
\begin{equation*}
  Lu = \begin{pmatrix} u(0) \\ u(1) \end{pmatrix}.
\end{equation*}
Its inverse is the Dirichlet mapping $D^{A,L}_\lambda: \R^2 \to
D(A_m)$
\begin{equation*}
  D^{A,L}_\lambda \phi = u(x) \in D(A_m):\qquad \begin{cases} (\lambda
  I - A_m) u(x) = 0, \\ L u = \phi. \end{cases}
\end{equation*}

\smallskip

As proposed in \cite{KrMuNa}, we define a {\em mild solution} of
(\ref{sistema}) a function $\u \in C([0,T];\Chi)$ such that
\begin{equation*}
  \begin{cases}
    u(t) = u_0 + A_m \int_0^t u(s) \, {\rm d}s, & t \in [0,T] \\
    v(t) = v_0 + B \int_0^t v(s) \, {\rm d}s + C \int_0^t u(s) \, {\rm d}s.
  \end{cases}
\end{equation*}

In order to use semigroup theory to study equation (\ref{sistema}), we
consider a matrix operator describing the evolution with feedback on
the boundary
\begin{equation*}
  \bA = \begin{pmatrix} A_m & 0 \\ C & B \end{pmatrix}
\end{equation*}
on the domain
\begin{equation*}
  D(\bA) = \{\u \in D(A_m) \times \R^{2} \,:\, L u = v\}.
\end{equation*}
Then a mild solution for equation (\ref{sistema}) exists if and only
if $\bA$ is the generator of a strongly continuous semigroup.

\smallskip

The above definition of the domain $D(\bA)$ puts in evidence the
relation between the first and the second component of the vector
$\u$. There is a different characterization that is sometimes useful
in the applications.

Let us define the operator $A_0$ as $A_0 = A_m$ on $D(A_0) = \{u \in
D(A_m) \,:\, L u = 0\}$. We can then write the domain of $\bA$ as
\begin{equation*}
  D(\bA) = \{\u \in D(A_m) \times \partial X \,:\, u - D^{A,L}_0 v \in D(A_0)\}.
\end{equation*}
The operator $\bA$ can be decomposed as the product
\begin{equation*}
  \bA = \begin{pmatrix} A & 0 \\ 0 & B \end{pmatrix} 
  \begin{pmatrix} I & - D^{A,L}_0 \\ C & I \end{pmatrix}
\end{equation*}
Then, according to {\tt Engel} \cite{En99}, $\bA$ is called a one-sided
$K$-coupled matrix-valued operator.

\subsection*{Proof of Theorem \ref{th:1}}

In this section we apply form theory in order to prove generation
property of the operator $\bA$, compare the monograph \cite{Ouh}.

\begin{proposition}\label{pr:bA-genera}
  $\bA$ is the infinitesimal generator of a strongly continuous,
  analytic semigroup of contractions, self-adjoint and compact.
\end{proposition}

We will give the proof in two steps. First of all we will consider the
following form:
\begin{equation*}\label{eq:forma}
  a(\u,\v)=\int_0^1 u'(x)v'(x) \, {\rm d}x + b_0\, u(0)\, v(0) + b_1\, u(1)\, v(1)
\end{equation*}
on the domain
\begin{displaymath}
  V=\left\{\u=(u,\alpha) \in H^1(0,1)\times \R^2\mid u(0)=\alpha_0,u(1)=\alpha_1\right\}
\end{displaymath}
and we will show that it is densely defined, closed, positive,
symmetric and continue.  Moreover, the operator associated with the
form $a$ is $(\bA,D(\bA))$ defined above. According to \cite{Ouh},
this implies that the operator $\bA$ is self-adjoint and generates a
contraction semigroup $\ee^{t\bA}$ on $\Chi$ that is analytic of angle
$\frac{\pi}{2}$.  Then we will show the self-adjointness and the
compactness of the semigroup $\ee^{t\bA}$. To see this, we will refer to
\cite{GGK}.
  
Let us begin with the properties of the form $a$.
  
\begin{lemma}
  The form $a$ is densely defined, closed, positive, symmetric and
  continue.
\end{lemma}
    
\begin{proof}
  By assumption, since $b_0$ and $b_1$ are positive real numbers, it
  follows that in particular $a$ is symmetric and positive.
      
  It is clear that $V$ is a linear subspace of $\Chi$. Observe that
  $V$ is dense in $\Chi$ if any $\u\in \Chi$ can be approximated with
  elements of $V$. Consider $(u,\alpha)\in L^2[0,1]\times \R^2$. Since
  $C^{\infty}_c[0,1]$ is dense in $L^2(0,1)$ it follows that for all
  $\varepsilon >0$ there exists $v\in C^{\infty}_c[0,1]$ such that
  \begin{equation*}
    \left|u-v\right|_{L^2[0,1]}\leq \dfrac{\varepsilon}{3}.
  \end{equation*}
  Now let $\rho_0(x)$ be a symmetric function in $C^{\infty}_c(\R)$
  with support in $B_{\varepsilon}(0)$, $\rho_0(0)=1$ and
  $\int_{\R}\rho_0(x) \, {\rm d}x = \varepsilon/3$. Finally, let
  $\rho_1(x)=\rho_0(x-1)$. Then, if we define the function $\rho = v +
  \alpha_0 \, \rho_0\Big|_{[0,1]} + \alpha_1 \, \rho_1\Big|_{[0,1]}$,
  we have:
  \begin{align*}
    \left|u-\rho\right|_{L^2[0,1]}&\leq
    \left|u-v\right|_{L^2[0,1]}
    +\left|\alpha_0\rho_0\right|_{L^2[0,1]}
    +\left|\alpha_1\rho_1\right|_{L^2[0,1]}\leq \\
    & \leq \max \left\{1,\alpha_0,\alpha_1\right\}\varepsilon.
  \end{align*} 
  Morever, $\rho(0)=\alpha_0$ and $\rho(1)=\alpha_1$. Thus
  \begin{equation*}
    \left|(u,\alpha)-(\rho,\rho(0),\rho(1))\right|_{\Chi}\leq M
    \varepsilon
  \end{equation*}
  for a suitable $M$.  This shows that $V$ is dense in $\Chi$.
      
  In order to check closedness and continuity of $a$, observe first
  that the norm induced by $a$ on the space $V$ is equivalent to the
  norm given by the inner product
  \begin{equation*}
      (\u,\v)_V=\int_0^1
      \left[u'(x)v'(x)+u(x)v(x)\right] \, {\rm d}x+u(1)v(1)+u(0)v(0). 
  \end{equation*}
  In fact, if we set  $b=b_0+b_1$, we have
  \begin{equation*}
    \left\|\u\right\|_a=\sqrt{a(\u,\u)+\left\|\u\right\|^2_V}
  \end{equation*}
  so that
  \begin{equation*}
    \left\|\u\right\|^2_a \leq 
    2\left\|u\right\|^2_{H^1(0,1)} + 2 b \left[u(0)^2+u(1)^2\right] \leq 
    \max \left\{2,2b \right\}\left\|\u\right\|^2_V.
  \end{equation*}
  Now observe that $V$ becomes a Hilbert space when equipped with the
  inner product defined above since $V$ is a closed subspace of
  $H^1(0,1)\times \R^2$. Then $a$ is closed.
      
  Finally, $a$ is continuous. To see this, take $\u,\v\in V$; then
  \begin{align*}
    \left|a(\u,\v)\right|&\leq \int_0^1
    \left|u'(x)v'(x)\right| \, {\rm d}x +
    b\left[\left|u(0)\right|\left|v(0)\right| +
      \left|u(1)\right|\left|v(1)\right|\right]  \\
    &\leq \left\|u\right\|_{H^1(0,1)}\left\|v\right\|_{H^1(0,1)} +
    b\left[\left|u(0)\right|\left|v(0)\right|+
      \left|u(1)\right|\left|v(1)\right|\right] \\
    & \leq \left\|\u\right\|_V\left\|\v\right\|_V\leq
    M\left\|\u\right\|_a\left\|\v\right\|_a
  \end{align*}
  by the Cauchy-Schwartz inequality.
\end{proof}

\begin{lemma}
  The operator associated with $a$ is $(\bA,D(\bA))$ defined above.
\end{lemma}
    
\begin{proof}
  Denote by $(\cC,D(\cC))$ the operator associated with $a$. By
  definition, $\cC$ is given by
  \begin{align*}
    D(\cC) &= \left\{\f \in V\mid \exists \g \in \Chi \: \text{s.t.}\:
      a(\f,\g) = (\g,\h)_{\Chi} \forall \h\in V \right\} \\
    \cC \f&=-\g.
  \end{align*}
      Let us first show that $\bA\subset \cC$. Take $\f\in D(\bA)$.
      Then for all $\h\in V$
      \begin{align*}
        a(\f,\h)&=\int_0^1 f'(x) h'(x) \, {\rm d}x + b_0 f(0) h(0)+ b_1 f(1) h(1) \\
        &= \left.f'(x)h(x)\right|_0^1-\int_0^1 f''(x) h(x) \, {\rm d}x +b_0 f(0) h(0) + b_1 f(1) h(1) \\
        &= f'(1)h(1)- f'(0)h(0)-\int_0^1
        f''(x) h(x) \, {\rm d}x + b_0 f(0) h(0) + b_1 f(1) h(1).
      \end{align*}
      At the same time, if we set
      $\alpha=(f(0),f(1))$, $\beta=(h(0),h(1))$, we have
      \begin{align*}
        \left(\bA\f,\h\right)&=(Af,h)_{L^2(0,1)}+(Cf+B\alpha,\beta)_{\R^2}=\\
        &=\int_0^1 f''(x) h(x) \, {\rm d}x + f'(0) h(0) - f'(1)h(1) \\
        &\quad - b_0 f(0) h(0) - b_1 f(1) h(1) = -a(\f,\g).
      \end{align*}
      The last equality shows that $\bA\subset \cC$.
      
      To check the converse inclusion $\cC\subset\bA$ take $\f\in
      D(\cC)$. By definition, there exists $\g\in \Chi$ such that
      \begin{align*}
        a(\f,\h)=(\g,\h)_{\Chi}, \quad \forall \h\in V
      \end{align*}
      that is,
      \begin{displaymath}
        \int_0^1 f'(x)h'(x) \, {\rm d}x = \int_0^1 g(x)h(x) \, {\rm d}x.
      \end{displaymath}
      Now choose $\h=(h,\alpha)\in V$ such that the function $h$ belongs to $H^1_0(0,1)$ 
      (the existence of such a function is ensured by the continuous 
      embedding of $H^1_0(0,1)$in $H^1(0,1)$). 
      Then by the last equality we cand derive that
      $f' \in H^1(0,1)$ and $g$ is the weak derivative
      of $f'$: it follows that $f'\in H^1(0,1)$ and we conclude that
      $f\in H^2(0,1)$.  Integrating by parts as in the proof of the
      first inclusion we see that
      \begin{align*}
        a(\f,\h) &= \int_0^1 f'(x)h'(x) \, {\rm d}x + b_0 f(0) h(0) +b_1 f(1) h(1) \\
        &=\left.f'(x)h(x)\right|_0^1 - \int_0^1 f''(x) h(x) \, {\rm d}x + b_0 f(0) h(0) + b_1 f(1) h(1) \\
        &=(-\bA\f,\h)=(\g,\h), \quad \forall \h\in V.
      \end{align*}
      This implies that $\bA\f=-\g$, and the proof is complete.
    \end{proof}

    \begin{corollary}
      The operator $(\bA,D(\bA))$ is self-adjoint and dissipative.
      Moreover it has compact resolvent.
    \end{corollary}

    \begin{proof}
      The self-adjointness of $\bA$ follows by
      \cite{Ouh} (Proposition 1.24) and he dissipativity is obsvious. Since
      $D(\bA)\subset H^2(0,1)\times \R^2$, the operator $\bA$ has
      compact resolvent and the claim follows.
    \end{proof}
    
    Taking into account the above corollary, it follows that $\bA$
    generates a contraction semigroup $(\ee^{t\bA})_{t\geq 0}$ on $\Chi$
    that is analytic of angle $\pi/2$ and self-adjoint. Finally, by
    \cite[Corollary XIX.6.3]{GGK} we obtain that $\ee^{t\bA}$ is compact for
    all $t>0$.
    
    Thus we have just proved Proposition \ref{pr:bA-genera}.
  
  \begin{remark}\label{decspet}
    By the Spectral Theorem \cite[Chapter XIX, Corollary 6.3]{GGK} it follows that there exists an
    orthonormal basis $\left\{e_n\right\}_{n\in \mathbb{N}}$ of $\Chi$
    and a sequence $\left\{\lambda_n\right\}_{n\in \mathbb{N}}$ of
    real negative numbers $\lambda_n\leq 0$, such that $e_n\in
    D(\bA)$, $\bA e_n=\lambda_ne_n$ and $\lim\limits_{n\to
      \infty}\lambda_n=-\infty$.  Moreover, $\bA$ is given by
    \begin{align*}
      \bA \u=\sum_{n=1}^{\infty} \lambda_n (\u,e_n)e_n, \quad \u\in D(\bA)
    \end{align*}
    and 
    \begin{equation*}
      \ee^{t\bA}\u=\sum_{n=1}^{\infty} e^{\lambda_n t} (\u,e_n)e_n, \quad
      \u \in \Chi.
    \end{equation*}
  \end{remark}


\subsection{Spectral properties of the matrix operator}

We shall now apply Theorem 2.5 in {\tt Engel}\cite{En99} in order
to describe the spectrum of $\bA$. According to that result
\begin{equation}\label{spettro}
  \sigma(\bA) \subseteq \sigma(A_0) \cup \sigma(B) \cup S
\end{equation}
where
\begin{equation}
  S = \{\lambda \in \rho(A_0) \cap \rho(B) \,:\, \Det(F(\lambda)) = 0\}.
\end{equation}
The matrix $F(\lambda)$ is defined as
\begin{equation*}
  F(\lambda) = I -(\lambda - B) L_\lambda K_\lambda R(\lambda,B)
\end{equation*}
where the operators $L_\lambda$ and $K_\lambda$ are given by
\begin{equation*}
  L_\lambda = -B R(\lambda,B) R(0,B) C, \qquad
  K_\lambda = -A_0 R(\lambda,A_0) D^{A,L}_0.
\end{equation*}
Notice that the matrix $F(\lambda)$ can also be written as
\begin{equation*}
  F(\lambda) = I + C A_0 R(\lambda,A_0) D^{A,L}_0 R(\lambda,B).
\end{equation*}

\begin{remark}
  In case when the feedback operator matrix $C$ is identically zero, the
  above construction implies that $S = \emptyset$.
\end{remark}

\subsection*{Determining the set $S$}
In the following, we construct explicitly the set $S$. The idea is
to construct the matrix $F(\lambda)$ and compute its determinant.

We have to distinguish two cases. If $\lambda<0$ we have
\begin{equation*}
  {\rm Det}(F(\lambda)) = 1 + \sqrt{-\lambda}
  \frac{\cos(\sqrt{-\lambda})}{\sin(\sqrt{-\lambda})}
  \left(\frac{1}{\lambda+b_0} + \frac{1}{\lambda+b_1}\right) +
  \frac{\lambda}{(\lambda+b_0)(\lambda+b_1)}
\end{equation*}

We note that the equation ${\rm Det}(F(\lambda))=0$ has infinite
 solutions $\{\lambda_j\}_{j \in \mathbb{N}}$ and every $\lambda_j$ belongs
 to the interval $(-\pi^2(j+1)^2,-\pi^2j^2)$.

Each $\lambda_j$ is eigenvalue of the operator $\bA$
corresponding to the eigenfunction $\phi_j=(e_j(x),e_j(0),e_j(1))$
where
$$e_j(x)=\frac{\sqrt{-\lambda_j}B_j}{b_0 +
\lambda_j}\cos\sqrt{-\lambda_j}x +B_j \sin\sqrt{-\lambda_j}x.$$
for a normalizing constant $0 < B_j <
\frac{1+\sqrt{-\lambda_j}}{-1+\sqrt{-\lambda_j}}$.

 If $\lambda > 0$ then
\begin{equation*}
{\rm Det}(F(\lambda)) = 1+ \sqrt{\lambda}
\frac{\left(1+\ee^{2\,\sqrt{\lambda}}\right)}{\big(-1+{{\ee}^{2\,{\sqrt{\lambda}}}}\big)\,
}\left( \frac{1}{b_0+\lambda} + \frac{1}{b_1+\lambda} \right) +
\frac{\lambda}{(b_0+\lambda)\, (b_1+\lambda)}.
\end{equation*}

We note that ${\rm Det}(F(\lambda))>0$ for every $\lambda >0$.
This means that there are not elements $\lambda$ strictly positive
in $S$. Moreover the eigenvalues of $\bA$ in $S$ are all negative.

\vskip1\baselineskip

\begin{remark}\label{R4}
  It is possible to verify directly with some computation that the
  eigenvalues of $A$ are not eigenvalues of $\bA$. 

  Further, the same happens in general with the eigenvalues of $B$,
  except in case $b_0$ and $b_1$ satisfy an explicit relation. In any
  case, also if $b_0$ and $b_1$ happen to belong to $\sigma(\bA)$,
  they are in a finite number and do not affect its behaviour.

  Therefore, with no loss of generality, in the following we may and
  do assume that all the eigenvalues of $\bA$ are contained in $S$.
\end{remark}

\vskip1\baselineskip

\begin{theorem}\label{th:sgr-hs}
  In the above assumptions
  the semigroup $\ee^{t\bA}$ is Hilbert-Schmidt, that is,
  \begin{equation}\label{HS}
    \sum_{i=1}^\infty |\ee^{t\bA}\phi_i|_{L^2(0,1)\times \R^2}^2 < \infty
  \end{equation}
  for any orthonormal basis $\{\phi_i\}$ of $L^2(0,1) \times \R^2$.
\end{theorem}

\begin{proof}
In order to prove that the semigroup $\ee^{t\bA}$ is
Hilbert-Schmidt, it is enough verify the (\ref{HS}) for an
orthonormal basis.
  Let $\{\phi_i\}$ the orthonormal sequence
of eigenfunctions of the operator $\bA$ described in Remark
\ref{decspet}.
Then
   \begin{equation*}
    \sum_{i=1}^\infty |e^{t\bA}\phi_i|_{L^2(0,1)\times \R^2}^2 = \sum_{i=1}^\infty e^{2t \lambda_i}
  \end{equation*}
  where $\lambda_i$ are the eigenvalues of the operator $\bA$.
  By (\ref{spettro}) it follows that
   \begin{equation*}
     \sum_{i=1}^\infty e^{2t \lambda_i}\leq  \sum_{i:\,\lambda_i \in \sigma(A)} e^{2t \lambda_i}
     + \sum_{i:\, \lambda_i \in \sigma(B)}e^{2t \lambda_i} + \sum_{i:\, \lambda_i \in S}e^{2t \lambda_i} .
  \end{equation*}
But, by Remark \ref{R4} we have that
 \begin{equation*}
     \sum_{i=1}^\infty e^{2t \lambda_i}\leq  \sum_{i:\,\lambda_i \in \sigma(B)}e^{2t \lambda_i} + \sum_{i :\, \lambda_i \in S}e^{2t
     \lambda_i}
  \end{equation*}
and the first of the last two series is a finite sum and the
second one converges since the eigenvalues ${\lambda}_i$ in $S$
are asymptotic to $-\pi^2i^2$.

\end{proof}

\section{The abstract problem}

In this section we are concerned with problem (\ref{eq:intro-1}): we
introduce the relevant assumptions and we formulate the main existence
and uniqueness result for its solution.

Let $\mathbf{W}=(W,V)$ be the Wiener process taking values in $=
L^2(0,1) \times \R^2$. We denote $\{{\calf}_t,\ t\in [0,T]\}$ the
natural filtration of $\mathbf{W}$, augmented with the family $\caln$
of $\Pr$-null sets of $\calf_T$:
$$
\calf_t=\sigma(\mathbf{W}(s)\; :\; s\in [0,t])\vee\caln.
$$
The filtration $\{\calf_t\}$ satisfies the usual conditions.

Define $\bF : [0,T]\times \cX \rightarrow \cX$ for every $\mathbf{u} =
\begin{pmatrix}u \\ v \end{pmatrix} \in \cX$
\begin{equation*}
  \bF(t,\mathbf{u})= \bF \left(t,\begin{pmatrix} u \\v \end{pmatrix}\right) =
  \begin{pmatrix}F(t,u) \\ 0 \end{pmatrix}
\end{equation*}
where $F(t,u)(\xi)=f(t, \xi,u(\xi))$.

Let $\bG$ be the mapping $[0,T] \times \cX \rightarrow L(\cX ,\cX)$
such that, for $\mathbf{u} = \begin{pmatrix}u \\ v \end{pmatrix}$ and
$\mathbf{y} = \begin{pmatrix}y \\ \eta \end{pmatrix}$ in $\cX$,
\begin{equation*}
  \bG\left(t,\begin{pmatrix}u \\ v \end{pmatrix}\right) \cdot \begin{pmatrix}y \\
    \eta \end{pmatrix} = \begin{pmatrix} G_1(t,u) \, y \\ G_2(t,v)
    \, \eta \end{pmatrix}
\end{equation*}
where
\begin{equation*}
  (G_1(t,u)\, y)(\xi) = g(t,\xi,u(\xi))y(\xi) \quad
  \text{and} \quad (G_2(t,v)\cdot \eta)= h(t) \, \eta;
\end{equation*}
we stress that $h$ is a diagonal matrix.

Therefore, we are concerned with the following abstract problem
\begin{equation} 
  \label{eq.astratta}
  \begin{cases}
    {\rm d}\textbf{u}_t = \bA \textbf{u}_t \, {\rm d}t+
    \bF(t,\textbf{u}_t)\, {\rm d}t+ \bG(t,\textbf{u}_t) {\rm
      d}\mathbf{W}_t \\ \textbf{u}_{t_0}=\textbf{u}_0
  \end{cases}
\end{equation}
on which we formulate the following assumptions.

\begin{assumption}\label{hp-A}
  \begin{enumerate}
  \item[]
  \item[(i)] $f:[0,T] \times [0,1] \times \R \rightarrow \R$, is a
    measurable mapping, bounded and Lipschitz continuous in the last
    component
    \begin{equation*}
      |f(t,x,u)| \leq K, \quad
      |f(t,x,u)- f(t,x,v)|\leq L |u-v|.\end{equation*}
    for every $t \in [0,T]$, $x \in [0,1]$, $u,v \in \R$.
  \item [(ii)] $g:[0,T] \times [0,1] \times \R \rightarrow \R$,
    is a measurable mapping such that
    \begin{equation*}
      |g(t,x,u)| \leq K, \quad |g(t,x,u)-g(t,x,v)|\leq L |u-v|
    \end{equation*}
    for every $t \in [0,T], \, x \in [0,1], \,u,v \in \R$.
  \item [(iii)] $h: [0,T] \rightarrow M(2,2)$ is a bounded measurable
    mapping verifying $|h(t)| \leq K$ for every $t \in [0,T]$.
  \end{enumerate}
\end{assumption}

The existence and uniqueness of the solution to (\ref{eq.astratta}) is
a standard result in the literature, see for instance the monograph
\cite{dpz:stochastic}. In order to apply the known results, we shall
verify that the nonlinear coefficients $\bF$ and $\bG$ satisfy
suitable Lipschitz continuous conditions. That will be enough to prove
the existence of a \emph{mild solution} which is a process
$\textbf{u}_t$ adapted to the filtration $\mathcal{F}_t$ satisfying
the following integral equation
\begin{equation} 
  \textbf{u}_t= e^{t \bA}\textbf{u}_0
  + \int_0^t e^{(t-s) \bA}\bF(s,\textbf{u}_s) \, {\rm d}s
  + \int_0^t e^{(t-s)\bA}\bG(s, \textbf{u}_s) \, {\rm d}\mathbf{W}s.
\end{equation}

\begin{proposition}\label{F-G}
  Under Assumptions \ref{hp-A}(i)--(iii), the following hold:
  \begin{enumerate}
  \item[1.] the mapping $\bF:\cX\to \cX$ is measurable and satisfies,
    for some constant $L>0$,
    \begin{equation*}
      |\bF(t,\mathbf{u})-\bF(t,\mathbf{v})|_{\cX}\leq
      L|\mathbf{u}-\mathbf{v}|_{\cX} \quad \mathbf{u},\mathbf{v} \in \cX.
    \end{equation*}
  \item[2.] $\bG$ is a mapping $[0,T] \times \cX \to L(\cX)$ such that
\begin{enumerate}
  \item[a.]
    for every $\mathbf{v}\in \cX$ the map $\bG(\cdot,\cdot)
    \mathbf{v}: [0,T] \times \cX\to \cX$ is measurable,
\item[b.]
    $e^{s\bA}\bG(t,\mathbf{u})\in L_2(\cX)$ for every $s>0$, $t \in
    [0,T]$ and $\mathbf{u}\in \cX$, and
\item[c.] for every $s>0$, $t \in
    [0,T]$ and $\mathbf{u}. \mathbf{v} \in \cX$ we have
    \begin{align}\label{primaipotesisug}
      &|e^{s\bA}\bG(t,\mathbf{u})|_{L_2(\cX)}\leq L\, s^{-1/4} \,
        (1+|\mathbf{u}|_{\cX}),
\\
\label{primaipotesisug-ii}
      &|e^{s\bA}\bG(t,\mathbf{u})-e^{s\bA}\bG(t,\mathbf{v})|_{L_2(\cX)}\leq L\, s^{-1/4} 
        |\mathbf{u}-\mathbf{v}|_{\cX},
\\
\label{secondaipotesisug}
      &|\bG(t,\mathbf{u})|_{L(\cX)}\leq L\;   (1+|\mathbf{u}|_{\cX}), 
    \end{align}
    for a constant $L>0$.
  \end{enumerate}
  \end{enumerate}
\end{proposition}

\begin{proof}

\begin{enumerate}
\item[1.] We have, for $\mathbf{u} = \begin{pmatrix}u \\ x
  \end{pmatrix}$ and $\mathbf{v} = \begin{pmatrix} v \\ y
  \end{pmatrix}$
  \begin{equation*}
    |\bF(t,\mathbf{u})-\bF(t,\mathbf{v})|_{\cX} = |F(t,u)-F(t,v)|_X \leq
    L |u-v|_X \leq L |\mathbf{u}-\mathbf{v}|_{\cX}.
  \end{equation*}

  
\item[2.] Condition (\ref{secondaipotesisug}) follows from the
  definition of $\bG$ and the Assumptions \ref{hp-A} (ii)-(iii) on $g$
  and $h$.
  
  Now we prove condition (\ref{primaipotesisug}). Let $\{\phi_k
  \}_{k \in \mathbb{N}}$ be an orthonormal basis in $\cX$. Then
  \begin{align*}
    |e^{s \bA} \bG(t,\mathbf{u})|_{L_2(\cX)}^2 &= \sum_{j,k} |<e^{s
      \bA} \bG(t,\mathbf{u})\phi_{j},\phi_{k}>|_{\cX}^2 \\ &=
    \sum_{j,k} |< \bG(t,\mathbf{u}) \phi_{j},e^{s \bA}
    \phi_{k}>|_{\cX}^2 \\ &\leq |\bG(t,\mathbf{u})|_{L(\cX)}^2 \,
    |e^{s \bA}|_{L_2(\cX)}^2 \leq L^2(1 + |\mathbf{u}|^2_{\cX})|e^{s
      \bA}|_{L_2(\cX)}^2.
  \end{align*}
  Using Theorem \ref{th:sgr-hs}, 
  \begin{equation*}
    |e^{s \bA}|_{L_2(\cX)}^2 \approx \sum_{n=1}^\infty \ee^{-2 s n^2} \approx \frac{1}{\sqrt{s}}
  \end{equation*}
  where $f(t) \approx g(t)$ means that $f(s)/g(s) = O(1)$ as $s \to 0$;
  this verifies (\ref{primaipotesisug}).
  
  In order to prove the last statement (\ref{primaipotesisug-ii}),
  we take the orthonormal basis $\{\phi_k \}_{k \in \mathbb{N}}$
  consisting of eigenvectors of $\bA$ (see Remark \ref{decspet}). We
  recall that $\phi_k =(e_k(x),e_k(0),e_k(1))$ where
  \begin{equation*}
    e_k(x)=B_k \frac{\sqrt{-\lambda_k}}{b_0 +
    \lambda_k}\cos\sqrt{-\lambda_k}x +B_k \sin\sqrt{-\lambda_k}x.
  \end{equation*}
  We have
  \begin{multline*}
    |e^{s \bA} \bG(t,\mathbf{u})-e^{s \bA}
    \bG(t,\mathbf{v})|_{L_2(\cX)}^2 = \sum_{j,k} |<e^{s \bA}[
    \bG(t,\mathbf{u})-\bG(t,\mathbf{v})]\phi_{j},\phi_{k}>|_{\cX}^2 \\ 
    = \sum_{j,k} |< \bG(t,\mathbf{u})-\bG(t,\mathbf{v}) \phi_{j},e^{s
      \bA} \phi_{k}>|_{\cX}^2 = \sum_k e^{2s \lambda_k}
    \,|\bG(t,\mathbf{u})-\bG(t,\mathbf{v})\phi_{k}|^2 .
  \end{multline*}
  But, for $\mathbf{u} = \begin{pmatrix}u \\ x \end{pmatrix}$ and
  $\mathbf{v} = \begin{pmatrix} v \\ y \end{pmatrix}$, by the
  definition of the operator $\bG$, we have
  \begin{multline*}
    |\bG(t,\mathbf{u})-\bG(t,\mathbf{v})\phi_{k}|_{\cX}^2 = \int_0^1
    |g(t,x,u(x))-
    g(t,x,v(x))|^2|e_k(x)|^2dx \\
    \leq \int_0^1 K^2|u(x)- v(x)|^2dx \leq
    K^2|\mathbf{u}-\mathbf{v}|_{\cX}^2
  \end{multline*}
  since the function $g$ is Lipschitz and $|e_k(x)| \leq B_k$ is
  uniformly bounded in $k$.  Consequently
  \begin{multline*}
    |e^{s \bA} \bG(t,\mathbf{u})-e^{s \bA} \bG(t,\mathbf{v})|_{L_2(\cX)}
    \leq \{\sum_{k} e^{2t
      \lambda_k}\}^{1/2}K|\mathbf{u}-\mathbf{v}|_{\cX} \\ \leq|e^{s
      \bA}|_{L_2(\cX)}K|\mathbf{u}-\mathbf{v}|_{\cX}
  \end{multline*}
  which concludes the proof.
\end{enumerate}
\end{proof}

\begin{proposition} Under the assumptions \ref{hp-A}
for every $p\in [2,\infty)$
  there
exists a unique process $\mathbf{u}\in L^p(\Omega;C([0,T];\cX))$
solution of (\ref{eq.astratta}).
\end{proposition}

\begin{proof} We can apply Theorem 5.3.1 in \cite{DaPZa}. In fact by
Proposition \ref{pr:bA-genera} the operator $\bA$ generates a
strongly continuous semigroup $\{e^{t\bA} \}$ of bounded linear
operators in the Hilbert space $\cX$. Moreover,
for this theorem to apply we need to verify that coefficients $\bF$ and $\bG$ satisfy  conditions
(\ref{primaipotesisug})---(\ref{secondaipotesisug}), which follows from Proposition
\ref{F-G}.
\end{proof}

\section{Stochastic control problem}

After some preliminaries, in this section we are concerned with an
abstract control problem in infinite dimensions. We settle the
problem in the framework of weak control problems (see \cite{FlSo}).
\linea

We aim to control the evolution of the system by the boundary.
This means that we assume a boundary dynamic of the form:
\begin{equation}\label{bcond.controllo}
  \partial_t v(t) = b v(t) -  \partial_\nu u(t,\cdot) + h(t) [z(t) +
  \dot V(t)]
\end{equation}
where $z(t)$ is the control process. We require that $z \in
L^2(\Omega \times [0,T];\R^2)$.

\linea

As in the previous section we can write the system
\begin{equation}\label{sistemacontrollo}
\begin{cases}
  \partial_t u(t,x) = \partial_x^2 u(t,x)
  + f(t,x,u(t,x)) + g(t,x,u(t,x)) \dot W(t,x)\\
   \partial_t v(t) = b v(t) -  \partial_\nu u(t,\cdot) + h(t) [z(t) +
  \dot V(t)]
  \end{cases}
\end{equation}
 in the following abstract form
\begin{equation}\label{state equation}
  {\rm d}\textbf{u}_t^z = \bA \textbf{u}_t^z \, {\rm d}t+ \bF(t,\textbf{u}_t^z) \, {\rm d}t +\bG(t,\textbf{u}_t^z)[ Pz_t \, {\rm d}t + {\rm
  d}\mathbf{W}_t] \qquad
\textbf{u}_{t_0}=\textbf{u}_0
\end{equation}
where $P: \R^2 \to \cX$ is the immersion of the boundary
space in the product space $\cX = X \times \R^2$.
Equation (\ref{state equation}), in the framework of stochastic
optimal control problem, is called the controlled state equation
associated to an admissible control system.  We recall that, in
general, fixed $t_{0}\geq 0$ and $u_{0}\in \cX$, an {\it
admissible control system} (a.c.s) is given by $(\Omega,\F,
\{\F_{t}\}_{t \geq 0}, \Pr, \{\mathbf{W}_t\}_{t \geq 0}, z)$ where
\begin{itemize}
\item $(\Omega, \F, \Pr)$ is a probability space,

\item $\{{F}_{t}\}_{t\geq 0}$ is a filtration in it, satisfying
  the usual conditions,

\item $\{\mathbf{W}_{t}\}_{ t\geq 0 }$ is a Wiener process with
values in $\cX$
  and adapted to the filtration $\{\F_{t}\}_{t \geq 0}$,

\item $z$ is a process with values in a space $K$, predictable
  with respect to the filtration $\{\F_{t}\}_{t \geq 0}$ and satisfies the
  constraint: $z(t)\in \cZ$, $\Pr$-a.s., for almost every $t\in
  [t_{0},T]$, where $\cZ$ is a suitable domain of $K$.
\end{itemize}

In our case the space $K$ coincide with $\R^2$.

\linea

To each a.c.s.\ we associate the mild solution $\mathbf{u}^z$ of
state equation the mild solution
   $\mathbf{u}^z\in C([t_{0},T];L^{2}(\Omega;\cX))$ of the state
   equation.
We introduce the functional cost
\begin{equation}\label{costo}
  J(t_0,u_0,z)=\E \int_{t_0}^T \lambda(s,\mathbf{u}_s^z,z_s)) \, {\rm d}s + \E
  \phi(\mathbf{u}_T^z)
\end{equation}

We consider the problem of minimizing the functional $J$
over all admissible control systems (which is known in the literature as the weak formulation of the control problem); any a.c.s. that minimize $J$ -if it exsts- is called optimal for the control problem.

We define in classical way the Hamiltonian function relative to
the above problem
$$\psi: [0,T] \times
\cX \times \cX \rightarrow \mathbb{R}$$ setting
\begin{equation}\label{definhamilton}\psi(t,\u,\w)= \inf_{z \in
\cZ}\{\lambda(t,\u,z) + <\w, Pz>\}
\end{equation} and we define he following set
$$\Gamma(t,\u,\w)= \{ z \in\mathcal{Z} : \lambda(t,\u,z) + <\w,
Pz >= \psi(t,\u,z)\}$$

We consider the Hamilton-Jacobi-Bellman equation associated to the
control problem
\begin{equation}\label{kolmogorovnonlineare}
  \begin{cases}\displaystyle
\frac{\partial v(t,x)}{\partial t}+\mathcal{L}_t [v(t,\cdot)](x) =
\psi (t, x,v(t,x),\bG(t,x)^*\nabla_{x}v(t,x)),\\\hskip 8cm t\in [0,T],\,
x\in \cX,\\ 
\displaystyle v(T,x)=\Phi(x).
\end{cases}
\end{equation}
where the operator $\mathcal{L}_t$ is defined by
$$
\mathcal{L}_t[\phi](x)=\frac{1}{2}{\rm Trace }\left(
\bG(t,x){\bG(x)}^*\nabla^2\phi(x)\right) + < \bA x,\nabla\phi(x)>.
$$

Under suitable assumptions, if we let $v$ denote the unique
solution of (\ref{kolmogorovnonlineare}) then we have
$J(t,x,z)\geq v(t,x)$ and the
 equality holds if and only if the following feedback law is verified
 by $z$ and $\mathbf{u}_{\sigma }^{z}$:
$$
z(\sigma)= \Gamma(\sigma,\mathbf{u}_{\sigma }^{z},
\bG(\sigma,\mathbf{u}_{\sigma
}^{z})^{*}\nabla_{x}v(\sigma,\mathbf{u}_{\sigma }^{z})).
$$
Thus, we can characterize optimal controls by a feedback law.

This class of stochastic control problems, in infinite dimensional
setting, has been studied by Fuhrman and Tessitore \cite{FT1} (We
refer to Theorem 7.2 in that paper
 for precise statements and additional results).

In order to characterize optimal controls by a feedback law we
have to require that the abstract operators $\bF$ and $\bG$
satisfy further regularity conditions.

We will prove that, under suitable assumptions on the functions
$f$ and $g$ in the problem (\ref{sistemacontrollo}), the abstract operators
fit the required conditions.

We impose that the operators $\bF$ and $\bG$ are G\^ateaux
differentiable. This notion of differentiability is weaker than
the differentiability in the Fr\'echet sense.

We recall that for a mapping $F:X\to V$, where $X$ and $V$ denote
Banach spaces, the directional
  derivative at point $x\in X$ in the direction $h\in X$
  is defined as
  $$
  \nabla F(x;h)=\lim_{s\to 0}\frac{F(x+sh)-F(x)}{s},
  $$
  whenever  the limit exists
  in the topology of $V$. $F$ is called G\^ateaux
differentiable at point $x$ if it has directional derivative in
every direction at point $x$ and there exists an element of
$L(X,V)$, denoted $\nabla F(x)$ and called G\^ateaux derivative,
such that $\nabla F(x;h)=\nabla F(x)h$ for every $h\in X$.
\begin{definition} 
We say that a mapping $F:X\to V$ belongs to the class
${\mathcal{G} }^1 (X;V)$ if it is continuous, G\^ateaux
differentiable on $X$, and $\nabla F:X\to L(X,V)$ is strongly
continuous.
\end{definition}
 The last requirement of the definition means that for every
$h\in X$ the map $\nabla F(\cdot)h:X\to V$ is continuous. Note
that  $\nabla F:X\to L(X,V)$ is not continuous in general if
$L(X,V)$ is endowed with the norm operator topology; clearly, if
this happens then $F$ is Fr\'echet differentiable on $X$.
 Membership of a map in $\mathcal{G}^1(X,V)$ may be conveniently
 checked as shown in the following lemma.
\begin{lemma}\label{staing}
  A map $F:X\to V$ belongs to $\mathcal{G}^1(X,V)$
  provided the following conditions hold:
\begin{itemize}
  \item[i)]  the directional
  derivatives $\nabla F(x;h)$ exist
  at every point $x\in X$ and in every direction $h\in X$;
  \item[ii)] for every  $h$, the mapping
   $\nabla F(\cdot;h): X\to V$
  is continuous;
  \item[iii)] for every $x$, the mapping $h\mapsto\nabla F(x;h)$
  is continuous from $X$ to $V$.
\end{itemize}
\end{lemma}

When $F$ depends on additional arguments,
 the previous definitions and properties  have
obvious generalizations.


The following assumptions are necessary in order to provide
G\^ateaux differentiability for the coefficients of the abstract
formulation.

\vskip 1\baselineskip

\begin{assumption}\label{hp-B}
For a.a. $t \in [0,T]$, $\xi \in [0,1]$ the functions
  $f(t,\xi, \cdot)$ and
  $g(t,\xi, \cdot)$ belong to the class $C^1(\R)$.
\end{assumption}

\vskip 1\baselineskip

\begin{proposition}\label{gateaux}
Under assumptions \ref{hp-A} and \ref{hp-B}, for every $s>0$,
$t\in [0,T]$,
$$
\bF(t,\cdot)\in \mathcal{ G}^1(\cX, \cX),\qquad
e^{s\bA}\bG(t,\cdot)\in \mathcal{ G}^1(\cX, L_2(\cX)).$$
\end{proposition}

\begin{proof}
The first statement is an immediate consequence of the fact
that $f(t,\xi, \cdot) \in C^1(\R,\R)$. In order to prove that $e^{s
\bA}\bG(t, \cdot)$ belongs to the class $\mathcal{G}^1(\cX,L_2(\cX))$
we use the continuous differentiability of $g$ and an argument similar
to that used in the proof of Proposition \ref{F-G}. 

We note that, for 
$\mathbf{u} = \begin{pmatrix} u \\ x \end{pmatrix}$ 
and 
$\mathbf{v} = \begin{pmatrix} v \\ y \end{pmatrix}$, 
the gradient operator $\nabla_{\u} \left(e^{s\bA}
  \bG(t,\mathbf{u})\right) \v $ is an Hilbert Schmidt operator that maps
$$
\mathbf{w}=\begin{pmatrix}w
\\ p \end{pmatrix}
\mapsto e^{s\bA} \begin{pmatrix} g_u(t,\cdot,u(\cdot)) w(\cdot) v(\cdot) \\ 0 \end{pmatrix} =e^{s\bA}\left(\nabla_u (\bG(t,\u)\v )(\w)\right)
$$ 

In fact, we have
\begin{align*}
\lim_{r\rightarrow 0} &\left\|\frac{e^{s\bA}\bG(t,\u+r\v)-e^{s\bA} \bG(t,\mathbf{u})}{r} -
\nabla e^{s\bA} \bG(t,\mathbf{u})\v\right\|_{L_2(\cX)}
\\
&=\lim_{r\rightarrow 0} \sum_{j,k} \left| <\frac{
e^{s\bA}\bG(t,\u+r\v)- e^{s\bA}\bG(t,\u)}{r}\, \phi_j - e^{s\bA}\left(\nabla_u (\bG(t,\u)\v )\phi_j \right), \phi_k>\right|_{\cX}^2
\\
&=\lim_{r\rightarrow 0} \sum_{j,k} \left|<\left(\frac{
\bG(t,\u+r\v)- \bG(t,\u)}{r} - \nabla_u \bG(t,\mathbf{u}) \v
\right)\phi_j,e^{s\bA} \phi_k>\right|_{\cX}^2
\\
&=\lim_{r\rightarrow 0} \sum_{k}e^{2s\lambda_k} \left|
\left(\frac{\bG(t,\u+r\v)- \bG(t,\u)}{r} - \nabla_u
\bG(t,\mathbf{u}) \v \right)\phi_k \right|_{\cX}^2 
\\
&= \lim_{r\rightarrow 0} \sum_{k}e^{2s\lambda_k}
 \int_0^1 \left|\frac{g(t,u(\xi)+rv(\xi))-g(t,u(\xi))}{r} e_k(\xi)- 
g_u(t,u(\xi))v(\xi)e_k(\xi) \right|^2d\xi
\\
&\leq c  \lim_{r\rightarrow 0} \sum_{k}e^{2s\lambda_k}
 \int_0^1 \left|\frac{g(t,u(\xi)+rv(\xi))-g(t,u(\xi))}{r} - 
g_u(t,u(\xi))v(\xi) \right|^2d\xi
\\
&=c\lim_{r\rightarrow 0} \sum_{k}e^{2s\lambda_k} \int_0^1 \left|
 \int_0^1 \big[g_u(t,u(\xi)+\alpha rv(\xi))- g_u(t,u( \xi))\big] d\alpha \, v(\xi) \right|^2d\xi
\end{align*}
and, by dominated convergence, this limit is equal to zero.
In similar way we can prove the points $(ii)-(iii)$ of Lemma \ref{staing} to obtain the thesis.
\end{proof}

In order to prove the main result of this section we require the
following hypothesis.

\begin{assumption}\label{hpcontrollo}
\begin{enumerate}
\item[]
\item[(i)] $\lambda$ is measurable and for a.e. $t \in [0,T]$, for
all
      $\textbf{u},\mathbf{u}' \in \cX$, $z \in \mathcal{Z}$
        \begin{equation*}
        |\lambda(t,\textbf{u},z) -
      \lambda(t,\textbf{u}',z)| \leq C|1 +\textbf{u} +
      \textbf{u}'|^m|\textbf{u}-\textbf{u}'|
        \end{equation*}
        \begin{equation*}
          |\lambda(t,0,z)| \leq C
        \end{equation*}
        for suitable $C \in \R^+$, $m \in\mathbb{N}$;
      \item[(ii)] $\mathcal{Z}$ is a Borel and bounded subset of
        $\R^2$;
      \item[(iii)] $\Phi \in {\mathcal{G}}^1 (\cX,\mathbb{R})$
      and, for every $\sigma \in [0,T]$,
        $\psi(\sigma, \cdot, \cdot) \in {\mathcal{G}}^{1,1} (\cX \times \cX, \R)$;
      \item[(iv)] for every $t \in [0,T]$, $\u,\w,\h \in \cX$
        \begin{equation*}
          |\nabla_{\u} \psi(t,\u,\w)\h| + |\nabla_{\u} \phi(\u)\h| \leq
          L|\h|(1 +|\u|)^m;
        \end{equation*}
      \item[(v)] for all $t \in [0,T]$, for all $\mathbf{u} \in \cX$ and $\w\in \cX$ there
        exists a unique $\Gamma(t,\mathbf{u},\w)\in \mathcal{Z}$ that
        realizes the minimum in (\ref{definhamilton}). Namely
        \begin{equation*}
          \lambda(t,\mathbf{u},\Gamma(t,\mathbf{u},\w)) + <\w, P
          \Gamma(t,\mathbf{u},\w)> = \psi(t,\mathbf{u},\w)
        \end{equation*}
  \end{enumerate}
\end{assumption}

\vskip1\baselineskip

\begin{theorem}\label{maincontrollo}
  Suppose that assumptions \ref{hp-A}, \ref{hp-B} and
  \ref{hpcontrollo} hold.  For all a.c.s. we have
  $J(t_{0},u_{0},z)\geq v(t_{0},u_{0})$ and the equality holds if and
  only if the following feedback law is verified by $z$ and $\u^{z}$:
  \begin{equation}
    z(\sigma)= \Gamma(\sigma,\u^{z}_{\sigma },
    G(\sigma,\u^{z}_{\sigma})^{*}\nabla_{x}v(\sigma,\u^{z}_{\sigma})),\quad
    \Pr- {\rm a.s. \; for\;a.a.\; } \sigma\in [t_{0},T].
    \label{leggefeedback} 
  \end{equation}
  Finally there exists at least an a.c.s. for which
  (\ref{leggefeedback}) holds. In such a system the closed loop
  equation:
  \begin{equation}\label{equazioneclosedloop}
    \begin{cases}
      d\overline{\u}_\tau =
    \bA\overline{\u}_\tau\; d\tau
    +\bG(\tau,\overline{\u}_\tau)P\Gamma(\tau, \overline{\u}_{\tau},
    \bG(\tau, \overline{\u}_{\tau })^*\nabla_{x}v(\tau ,
    \overline{\u}_{\tau }))\; d\tau
    \\ 
    \phantom{d\overline{\u}_\tau =
    \bA\overline{\u}_\tau\; d\tau
    }+\bF(\tau,\overline{\u}_\tau)\;
    d\tau+\bG(\tau,\overline{\u}_\tau)\; d\mathbf{W}_\tau,\qquad
    \tau\in [t_{0},T],
    \\ 
    \overline{\u}_{t_{0}} = \u_{0}\in \Chi.
    \end{cases}
  \end{equation}
  admits a solution and if $\overline{z}(\sigma)=
  \Gamma(\sigma,\overline{\u}_{\sigma} ,G(\sigma,\overline{\u}_{\sigma
  })^{*}\nabla_{x}v(\sigma , \overline{\u}_{\sigma }))$ then the
  couple $(\overline{z},\overline{\u})$ is optimal for the control
  problem.
\end{theorem}

\begin{proof}
By Proposition \ref{pr:bA-genera} we know that $\bA$ generates a
strongly continuous semigroup of linear operators $e^{t\bA}$ on
$\cX$. The assumption \ref{hp-A} ensures that the statements in
Proposition \ref{F-G} hold. Moreover the assumption \ref{hp-B}
guarantees that the results in Proposition \ref{gateaux} are true.
Finally these conditions together with the assumption
\ref{hpcontrollo} allow us to apply Theorem 7.2 in \cite{FT1} and
to perform the synthesis of the optimal control.
\end{proof}

\section*{References}


\begin{thebibliography}{15}

\bibitem{bertini/noja/posilicano}%
  M. Bertini, D. Noja, A. Posilicano, \emph{Dynamics and Lax-Phillips
    scattering for generalized Lamb models}, J. Phys. A:
  Math. Gen. \textbf{39} (2006), 15173--15195

\bibitem{chueshov/schmalfuss}%
  Igor Chueshov, Bj\"orn Schmalfuss, \emph{Parabolic stochastic
    partial differential equations with dynamical boundary conditions},
  Differential Integral Equations \textbf{17} (2004), no. 7-8,
  751--780.

\bibitem{dpz:stochastic} 
  Giuseppe Da Prato, Jerzy Zabczyk, \textbf{Stochastic equations in
    infinite dimensions}, Encyclopedia of Mathematics and its
  Applications, 44. Cambridge University Press, Cambridge, 1992.

\bibitem{DaPZa} 
  G. Da Prato, J. Zabczyk, {\bf Ergodicity for infinite-dimensional
    systems}, London Mathematical Society Lecture Notes Series, 229,
  Cambridge University Press, 1996.

\bibitem{debussche/fuhrman/tessitore}
  A. Debussche, M. Fuhrman, G. Tessitore, \emph{Optimal Control of a
    Stochastic Heat Equation with Boundary-noise and Boundary-control},
  to appear in ESAIM Control, Optimisation and Calculus of Variations.

\bibitem{En99} 
  K.-J. Engel, \emph{Spectral theory and generator property for
    one-sided coupled operator matrices}, Semigroup Forum \textbf{58}
  (1999), 267--295.

\bibitem{FlSo} W. H. Fleming, H. M.  Soner,
{\bf Controlled Markov processes and viscosity solutions},
Springer-Verlag, 1993.

\bibitem{FT1} 
  M. Fuhrman, G. Tessitore, \emph{Non linear Kolmogorov equations in
    infinite dimensional spaces: the backward stochastic differential
    equations approach and applications to optimal control},
  Ann. Probab. \textbf{30} (2002), no. 3: 1397-1465.

\bibitem{GGK} 
  Israel Gohberg, Seymour Goldberg, Marinus A. Kaashoek,
  \textbf{\small Classes of linear operators}, Vol.\ I, Birkh�ser
  Verlag, Basel, 1990. Operator Theory: Advances and Applications, 49.

\bibitem{KrMuNa}
  Marjeta Kramar, Delio Mugnolo, Rainer Nagel, \emph{Semigroups for
    initial-boundary value problems}.In: \textbf{\small Evolution equations:
    applications to physics, industry, life sciences and economics (Levico
    Terme, 2000)}, 275--292, Progr. Nonlinear Differential Equations
  Appl., 55, Birkhäuser, Basel, 2003.

\bibitem{maslowski} 
  Bohdan Maslowski, \emph{Stability of semilinear equations with
    boundary and pointwise noise}, Ann. Scuola Norm. Sup. Pisa
  Cl. Sci. (4) \textbf{22} (1995), no. 1, 55--93.


\bibitem{Mu06} 
  D. Mugnolo, \emph{Asymptotics of semigroups generated by operator
    matrices}, Ulmer seminare \textbf{10} (2005), 299--311.


\bibitem{Ouh}
  El Maati Ouhabaz, \textbf{Analysis of heat equations on domains},
  London Mathematical Society Monographs Series, 31. Princeton
  University Press, Princeton, NJ, 2005.


\end{thebibliography}
\end{document}